\newcommand{\NN} {\mathbb N}

\newcommand{\RR} {\mathbb R}
\newcommand{\CC} {\mathbb C}

\def \epsilon{\varepsilon}

\def \D  {{\mathcal D}}

\def \I  {{\mathcal I}}
\def \J {{\mathcal J}}
\def \LL {{\mathcal L}}

\def \S  {{\mathcal S}}

\def \d {\text{d}}
\renewcommand{\SS}  {{\mathcal S}}

\def\bfell{{\boldsymbol{\ell}}}
\def\bfF{{\boldsymbol{F}}}
\def\bfm{{\boldsymbol{m}}}
\def\bfn{{\boldsymbol{n}}}
\def\bfs{{\boldsymbol{s}}}

\def \bfw {{\boldsymbol{w}}}

\def \half {{\textstyle{ 1\over 2}}}

\def \Ga {\Gamma}
\def \si {\sigma}

\renewcommand{\S}{{\mathcal S}}

\documentclass[12pt,reqno]{amsart}

\usepackage{amsmath,amssymb}
\newtheorem{theorem}{Theorem}
\newtheorem{lem}{Lemma}

\newtheorem{cor}{Collorary}
\usepackage{url} 
\usepackage{xcolor}

\numberwithin{equation}{section}

\begin{document}
\title[]{Multiple standard twists of $L$-functions}
\author[]{J.Kaczorowski \lowercase{and} A.Perelli}
\maketitle

\medskip
%%%%%
{\bf Abstract.} {\small The standard twist of $L$-functions plays a fundamental role in the Selberg class theory. It is defined as an absolutely convergent Dirichlet series and admits meromorphic continuation beyond the half-plane of absolute convergence. Nowadays, the analytic properties of the standard twist $F(s,\alpha)$ of an $L$-function $F$ are well-understood. For example, it has poles when the positive number $\alpha$ belongs to the so-called spectrum of $F$, and is entire otherwise. In this paper, for a given set $\bfF=\{F_1,\dots,F_N\}$ of $L$-functions and $\bfs\in\CC^N$, we consider the multiple standard twist $\bfF(\bfs,\alpha)$. This is defined initially on a certain half-space of $\CC^N$, and we describe its meromorphic continuation to the whole space. Results in the multidimensional case are, in many ways, analogous to those in the one-dimensional case. In particular, the spectrum of a multiple standard twist is relevant to the description of the set of poles of $\bfF(\bfs,\alpha)$. There are also significant differences; for instance, in the structure of the singularities.}

\medskip
%%%%

\medskip
%%%%%

\smallskip
%%%%%
\section{Introduction}
%%%%%%%%%%%%%%%%%%%%%%%%%%%%%%%%%%%%%%%%%%%%%%%%%%%%%%%%%%%%%
\smallskip
%%%%%
Let 
\begin{equation}
\label{Dseries}
F(s)=\sum_{n=1}^\infty\frac{ a(n)}{n^s} \qquad (\si>1)
\end{equation}
be a function of degree $d_F>0$ from the extended Selberg class $\SS^\sharp$. For reader's convenience, we collect some basic definitions and notation in the next section. For $\alpha>0$, the standard twist of $F$ is defined for $\si>1$ by the absolutely convergent Dirichlet series
\[ 
F(s,\alpha)= \sum_{n=1}^\infty\frac{a(n)}{n^s}e(-\alpha n^{1/d_F}),
\]
where $e(x)=\exp(2\pi i x)$. We refer to \cite{Ka-Pe/2005} and \cite{Ka-Pe/2021a}  for an in-depth study of $F(s,\alpha)$. Here we recall only that there exists an infinite countable set ${\rm Spec }(F)\subset (0, \infty)$, called the spectrum of $F$, with the following properties. If $\alpha\not\in {\rm Spec }(F)$ then $F(s,\alpha)$ is an entire function, otherwise it is meromorphic on $\CC$ with at most simple poles at the points
\begin{equation}
\label{s*ell}
 s^*_\ell= \frac{d_F+1}{2d_F}-\frac{\ell}{d_F} - i\theta_F \qquad (\ell=0,1,2,\dots),
 \end{equation}
where $\theta_F$ denotes the internal shift of $F(s)$.

\smallskip
%%%%%
In this paper we develop a multidimensional version of the above theory. Let $\alpha>0$, $N\geq1$, $\bfF=\{F_1,\dots, F_N\in\S^\sharp\}$ with degrees $d_1,\dots,d_N>0$ and Dirichlet coefficients $a_1(n),\dots,a_N(n)$, and let
\[
\CC_1^N=\{ \bfs=(s_1,\dots,s_N)\in \CC^N, \Re(s_\nu)>1 \ \text{for} \ \nu=1,\dots,N\}.
\] 
Then for $\kappa_1,\dots,\kappa_N>0$ such that
\begin{equation}
\label{sumdkappa} 
\sum_{\nu=1}^N d_\nu\kappa_\nu = 1
\end{equation}
we consider the {\it multiple standard twist} $\bfF(\bfs,\alpha)$ of $F_1,\dots, F_N$, defined for $s\in\CC_1^N$ as
\[ 
\bfF(\bfs,\alpha) =\sum_{n_1=1}^\infty\ldots\sum_{n_N=1}^\infty
\frac{a_1(n_1)\cdots a_N(n_N)}{n_1^{s_1}\cdots n_N^{s_N}} e(-\alpha n_1^{\kappa_1}\cdots n_N^{\kappa_N}).
\]
The series is absolutely convergent, thus $\bfF(\bfs,\alpha)$ is holomorphic on the half-space $\CC_1^N$.  Clearly, when $N=1$ and $F_1=F$ we have that $\bfF( \bfs, \alpha)=F(s,\alpha)$. Note that for $N\geq 2$ and fixed $F_1,\dots,F_N$ there are infinitely many possibilities of choosing exponents $\kappa_1,\dots,\kappa_N>0$ satisfying \eqref{sumdkappa}. Hence, there are infinitely many multiple standard twists associated with the same set $\bfF$.  Formally, we should stress this fact in the notation, but for simplicity we stay with the above lighter notation.

\smallskip
%%%%%
We define the {\it spectrum} of a multiple standard twist as
\[
{\rm Spec}(\bfF) = \left\{ \prod_{\nu=1}^N \left(\frac{n_\nu}{q_\nu\kappa_\nu^{d_\nu}}\right)^{\kappa_\nu} : \prod_{\nu=1}^N a_\nu(n_\nu)\neq 0\right\},
 \]
 where $q_\nu=q_{F_\nu}$ denote the corresponding conductors. Moreover, let 
\begin{equation}
\label{new1}
d= \sum_{\nu=1}^N d_\nu \quad \text{and}\quad \theta =\sum_{\nu=1}^Nd_\nu\theta_{\nu},
\end{equation}
where $\theta_\nu=\theta_{F_\nu}$ denote the corresponding internal shifts; see \eqref{omegaxi}. Finally, for integers $\ell\geq 0$ we define the hyperplanes $H^*_\ell$ in $\CC^N$ as
\begin{equation}
\label{H}
H^*_\ell=\Big\{\bfs=(s_1,\ldots,s_N)\in \CC^N: \sum_{\nu=1}^Nd_\nu s_\nu=\frac{d+1}{2}-\ell -i\theta.  \Big\}.
\end{equation}
Again, when $N=1$ and $F_1=F$ we have that ${\rm Spec}(\bfF) ={\rm Spec}(F)$ and $H^*_\ell=s^*_\ell$. Note that, for $N\geq2$, Spec$(\bfF)$ depends on the choice of the parameters $\kappa_1,\dots,\kappa_N$ in \eqref{sumdkappa}, but the hyperplanes $H_\ell^*$ are independent from such parameters. With the above notation, the basic analytic properties of $\bfF( \bfs, \alpha)$ are described by the following theorem.

 %%%%%%%%%%%%%%%%%%%%
 \begin{theorem}
 \label{T1}
For $N\geq2$ we have 

\noindent
{\rm (i)} if $\alpha\not\in {\rm Spec}(\bfF)$, then $\bfF( \bfs, \alpha)$ has analytic continuation to an entire function on $\CC^N$;

\noindent
{\rm (ii)} if $\alpha\in {\rm Spec}(\bfF)$, then it has analytic continuation to a holomorphic function on 
\[
\CC^N \backslash \bigcup_{\ell=0}^\infty H_\ell^*,
\]
and for every $\ell\geq 0$ and $\bfs\in H^*_\ell$ the limit
\begin{equation}\label{limitXiell}
\Xi_\ell(\bfs,\alpha)=\lim_{\substack{\bfw\to\bfs\\ \bfw\not\in H^*_\ell}}\Big(\sum_{\nu=1}^Nd_\nu (w_\nu-s_\nu)\Big) \bfF(\bfw,\alpha)
\end{equation}
exists and is holomorphic on $H^*_\ell$.
\end{theorem}
 %%%%%%%%%%%%%%%%%%%%
Denoting as usual by $\Ga$ the Euler function, we have the following immediate corollary.
%%%%%%%%%%%%%%%%%%%%
\begin{cor}
\label{C1}
For $\alpha\in {\rm Spec}(\bfF)$ the quotient
\[
\frac{\bfF(\bfs,\alpha)}{\Gamma\Big( \sum_{\nu=1}^Nd_\nu s_\nu -\frac{d+1}{2}+i\theta\Big)}
\]
is an entire function on $\CC^N$. In particular, $\bfF(\bfs,\alpha)$ is meromorphic on $\CC^N$.
\end{cor}
 %%%%%%%%%%%%%%%%%%%%
The hyperplanes $H^*_\ell$, $\ell\geq 0$, are the `suspected' polar sets of $\bfF(\bfs,\alpha)$. As already remarked, in the one-dimensional case such hyperplanes reduce to the points $s_\ell^*$, and there is a pole at $s_\ell^*$ if and only if the corresponding structural invariant $d_{F}(\ell)$ is non-zero; see Theorem 3 in \cite{Ka-Pe/2021a}. In particular, we have always a pole at $s=s_0^*$. For the Riemann zeta function, this is the only singularity.  For other $L$-functions, there can be infinitely many poles, see \cite{Ka/2020}, or just a finite number of them, see \cite{Ka-Pe/2020}. The case $N\geq 2$ is different: all hyperplanes $H^*_\ell$ are polar sets of $\bfF(\bfs,\alpha)$, for every choice of the parameters $\kappa_1,\dots,\kappa_N$ in \eqref{sumdkappa}.

\smallskip
%%%%%
Before formulating a theorem describing the situation, we introduce further notation. For $\bfn=(n_1,\ldots,n_N)$ and $\kappa_1,\dots,\kappa_N$ satisfying \eqref{sumdkappa} we write 
\begin{equation}
\label{Cn}
 \alpha_\bfn= \prod_{\nu=1}^N \left(\frac{n_\nu}{q_{\nu} \kappa_\nu^{d_\nu}}\right)^{\kappa_\nu},
 \end{equation}
and for $\alpha\in {\rm Spec}(\bfF)$ we let
 \begin{equation}
 \label{Xi}
\Xi(\bfs, \alpha) = \omega_\bfF \sum_{n_1\geq 1}\dots \sum_{\substack{n_N\geq 1 \\ \hspace{-1.4cm} \alpha_\bfn= \alpha}} \, \prod_{\nu=1}^N\frac{\overline{a_\nu(n_\nu)}}{n_\nu^{1-s_\nu}},
\end{equation}
\begin{equation}
\label{omegatimes}
\omega_\bfF=\prod_{\nu=1}^N \omega_\nu,
\end{equation}
where $\omega_\nu=\omega_{F\nu}$ denote the corresponding root numbers, see \eqref{omegaxi}. Moreover, let $\xi_\nu=\xi_{F_\nu}$ denote the corresponding $\xi$-invariants, see again \eqref{omegaxi}. Clearly, the sums over the $n_\nu$ in \eqref{Xi} are all finite, hence $\Xi(\bfs, \alpha)$ is an entire function of $\bfs\in\CC^N$. With the above notation, an explicit expression for the limit functions $\Xi_\ell(\bfs,\alpha)$ in \eqref{limitXiell} is given by the following theorem.
%%%%%%%%%%%%%%%%%%%%
\begin{theorem}
\label{T2}
Let $N\geq 2$ and $\alpha\in {\rm Spec}(\bfF)$.  Then for every $\ell\geq 0$ we have
\begin{equation}
\label{Xiell}
\Xi_\ell(\bfs,\alpha)= \frac{1}{\sqrt{2\pi}}e^{- i \sum_{\nu=1}^{N} \big(\frac{\pi}{2} d_\nu\, s_\nu + \frac{\pi}{2}\,\xi_\nu + d_\nu\,\theta_\nu\,\log(d_\nu\,\kappa_\nu)\big)} \prod_{\nu=1}^N \Big(\frac{\kappa_\nu q_\nu^{1/d_\nu}}{2\pi}\Big)^{d_\nu\big(\frac{1}{2}-s_\nu\big)} W_\ell(\bfs) \Xi(\bfs,\alpha),
\end{equation}
where $W_\ell(\bfs)\in \CC[\bfs]$ has degree $2\ell$. Moreover, for every $\ell\geq 0$ the function $\Xi_\ell(\bfs,\alpha)$ does not vanish identically on $H_\ell^*$. In particular, every hyperplane $H_\ell^*$ is a polar set of $\bfF(\bfs,\alpha)$. 
\end{theorem}
%%%%%%%%%%%%%%%%%%%%
The above results can also be viewed in the framework of the theory of multiple Dirichlet series. Such a theory, initiated by Siegel \cite{Sie/1956}, in recent years has been greatly developed by Dorian Goldfeld and his school, both from the theoretical and the applications viewpoint. We refer to the papers by Goldfeld-Hoffstein \cite{Go-Ho/1985}, Diaconu-Goldfeld-Hoffstein \cite{D-G-H/2003}, Chinta-Friedberg-Hoffstein \cite{C-F-H/2006} and Bump \cite{Bum/2012} for a panoramic view on multiple Dirichlet series and a glimpse on applications. We conclude remarking that the ideas and techniques involved in our proofs differ from those in the multiple Dirichlet series theory.

\medskip
%%%%%
 {\bf Acknowledgements}.  We would like to thank the referee for reading the paper very carefully and for correcting many inaccuracies, leading to a definite improvement of the presentation. This research was partially supported by the GNAMPA group of the Istituto Nazionale di Alta Matematica, by grant PRIN 2022 {\sl ``The arithmetic of motives and L-functions''} and by grant 2021/41/B/ST1/00241 from the National Science Centre, Poland.

\medskip
%%%%%
\section{Definitions and lemmas} 
%%%%%%%%%%%%%%%%%%%%%%%%%%%%%%%%%%%%%%%%%%%%%%%%%%%%%%%%%%%%
The extended Selberg class $\S^\sharp$ was introduced in \cite{Ka-Pe/1999} and consists of non-identically vanishing Dirichlet series as in \eqref{Dseries} admitting a finite order meromorphic continuation to $\CC$ with the only possible singularity at $s=1$ and satisfying a general Riemann type functional equation of the form
\[
Q^s\prod_{j=1}^r\Gamma(\lambda_j s+\mu_j) F(s) = \omega
Q^{1-s}\prod_{j=1}^r\Gamma\big(\lambda_j(1- s)+\overline{\mu_j}\big) \overline{F}(1-s),
\]
where $\overline{F}(s) = \overline{F(\overline{s})}$, $Q, \lambda_1,\dots,\lambda_r$ are positive parameters, whereas  $\omega,\mu_1,\dots,\mu_r\in\CC$ with
$|\omega|=1$ and $\Re(\mu_j)\geq 0$ for $1\leq j\leq r$. We refer to our survey papers \cite{Kac/2006},\cite{Ka-Pe/1999b},\cite{Ka-Pe/2022a},\cite{Per/2005},\cite{Per/2004} for information on the analytic properties of the functions $F\in\SS^\sharp$. 
Degree $d_F$, conductor $q_F$, root number $\omega_F$, $\tau$-invariant $\tau_F$, $\xi$-invariant $\xi_F$ and internal shift $\theta_F$ of $F\in\SS^\sharp$ are defined, respectively, as
\[
d_F = 2\sum_{j=1}^r\lambda_j , \quad q_F= (2\pi)^{d_F}Q^2\prod_{j=1}^r\lambda_j^{2\lambda_j},
\]
\begin{equation}
\label{omegaxi}
\omega_F = \omega \prod_{j=1}^r \lambda_j^{-2i\Im\mu_j}, \quad \tau_F=\max_{1\leq j\leq r}\frac{\Im(\mu_j)}{\lambda_j},
  \quad \xi_F = 2\sum_{j=1}^r(\mu_j-\half) = \eta_F + i\theta_Fd_F,
\end{equation}
with $\eta_F,\theta_F\in\RR$. The spectrum of $F$ with Dirichlet coefficients $a(n)$ is defined by 
\[
\text{Spec}(F)= \left\{d_F\big(\frac{n}{q_F}\big)^{1/d_F}: a(n)\neq0 \right\}.
\]
The functional equation of $F\in \SS^\sharp$ with $d_F>0$ can be written in the following asymmetric invariant form
\begin{equation}
\label{EFA}
F(s) =  \omega_F S_F(s) h_F(s)\overline{F}(1-s),
\end{equation}
where
\[
S_F(s) = 2^r \prod_{j=1}^r \sin(\pi(\lambda_js + \mu_j)) 
\]
and
\[
h_F(s) =(2\pi)^{-r} \Big(\prod_{j=1}^r \lambda_j^{2i\Im\mu_j}\Big) Q^{1-2s} \prod_{j=1}^r \big(\Gamma(\lambda_j(1-s)+\overline{\mu}_j) \Gamma(1-\lambda_js-\mu_j)\big).
\]
Note that applying \eqref{EFA} twice we obtain 
\begin{equation}
\label{Shs1-s}
S_F(s)h_F(s)\overline{S_F}(1-s) \overline{h_F}(1-s)=1.
\end{equation}
%%
%%%%%%%%%%%%%%%%%%%%
\begin{lem}
\label{hasym}
Let $F\in\SS^\sharp$ with $d_F> 0$. Then for $M\geq 0$ and $|\arg(-s)|<\pi -\delta$ with any fixed $0<\delta<\pi$ we have
\[
\begin{split}
h_F(s) &= \frac{1}{\sqrt{2\pi}} \left(\frac{q_F^{1/d_F}}{2\pi d_F}\right)^{d_F(\frac12-s)} \sum_{\ell=0}^M d_F(\ell) \Gamma\big(d_F(s_\ell^*-s)\big) \\
&\hspace{3.5cm} + O\Big(\frac{1}{|s|}|\Ga(d_F(s_M^*-s))|\Big)
\end{split}
\]
with $s^*_\ell$ as in \eqref{s*ell}, certain coefficients $d_F(\ell)$ and an implicit constant depending on $F$ and $M$. In particular, taking $M=0$ we have
\[
h_F(s)\asymp  \left(\frac{q_F^{1/d_F}}{2\pi d_F}\right)^{-d_F\si}|\Gamma\big(d_F(s_0^*-s)\big)|.
\]
\end{lem}
%%%%%%%%%%%%%%%%%%%%
%%%%%%%%%%%%%%%%%%%%
\begin{proof}
This lemma follows from the estimates in Section 3.2 of \cite{Ka-Pe/2021a}. 
\hfill
\end{proof}
%%%%%%%%%%%%%%%%%%%%
The coefficients $d_F(\ell)$ are called the structural invariants of $F$.
%%%%%%%%%%%%%%%%%%%%
\begin{lem}
\label{prod}
For $1\leq\nu\leq N$, let $\alpha_\nu, r_\nu\in\CC$ and $A_\nu\in\RR$ be such that $|\alpha_\nu|, |r_\nu|\leq A_\nu$.  Then
\[ 
\left|\prod_{\nu=1}^N (\alpha_\nu+r_\nu) - \prod_{\nu=1}^N \alpha_\nu\right| \leq (2^N-1)\max_{1\leq\nu\leq N} \big( |r_\nu| \prod_{\substack{1\leq\mu\leq N\\ \mu\neq\nu}} A_\mu\big).
\]
\end{lem}
%%%%%%%%%%%%%%%%%%%%
\begin{proof}
By induction.
\hfill
\end{proof}
%%%%%%%%%%%%%%%%%%%%
%%%%%%%%%%%%%%%%%%%%
\begin{lem}
\label{lemma3}
For $M\geq 1$ and $|\arg(z)|\leq \pi - \delta$ with any fixed $0<\delta<\pi$ we have
\begin{equation}
\label{St1}
\log\Gamma(z+s) = (z+s-\frac12)\log z - z + \frac12\log2\pi 
+ \sum_{j=1}^M \frac{(-1)^{j+1} B_{j+1}(s)}{j(j+1)} \frac{1}{z^j} +O\big(\frac{1}{|z|^{M+1}}\big),
\end{equation}
where $B_{j+1}(s)$ are the Bernoulli polynomials. In particular, for every fixed $A<B$ we have
\begin{equation}
\label{St2}
\Ga(x+iy)\ll e^{-\frac{\pi}{2}|y|} |y|^{x-\frac{1}{2}}
\end{equation}
uniformly for $A<x<B$ and $|y|\geq 1$.
\end{lem}
%%%%%%%%%%%%%%%%%%%%
%%%%%%%%%%%%%%%%%%%%
\begin{proof}
This is the well-known Stirling formula; see for instance \cite{EMOT/1953-I}.
\hfill
\end{proof}
%%%%%%%%%%%%%%%%%%%%
We denote by  $(z)_{m}$  the Pochhammer symbol, defined by
\[ 
(z)_{m}=\begin{cases}
1 & \text{if $m=0$}\\
z(z+1)\ldots(z+m-1) & \text{if $m>0.$} 
\end{cases}
\]
We need the following result, which is a modified version of the Lemma in the Appendix in \cite{Ka-Pe/2022a}.
%%%%%%%%%%%%%%%%%%%%
\begin{lem}
\label{SumXT}
Let $A,b\in\CC$ with $A\neq0$ be fixed. Then for every $M\geq 1$ there exist linear forms
\[ 
R_{m}(X_1,\ldots,X_{m})\quad \text{and} \quad \widetilde{R}_{m}(X_1,\ldots,X_{m-1})   \qquad (1\leq m \leq M)
\]
satisfying
\begin{equation}
\label{Rm}
R_{m}(X_1,\ldots,X_{m})= (-1)^m A^m X_m + \widetilde{R}_m(X_1,\ldots,X_{m-1})
\end{equation}
and a polynomial $Q_M(X_1,\ldots,X_M, T)$ with
\[
\deg_{X_j}Q_M(X_1,\ldots,X_M,T)\leq 1\quad \text{for $j=1,\dots,M$ and} \quad\deg_TQ_M(X_1,\ldots,X_M,T)\leq M-1
\]
such that 
\begin{equation}
\label{R}
\sum_{k=1}^M \frac{X_k}{T^k} =\sum_{m=1}^M \frac{(-1)^{m} R_{m}(X_1,\ldots,X_{m})}{(AT+b)_{m}}
+\frac{Q_M(X_1,\ldots,X_M,T)}{T^M(AT+b)_M}.
\end{equation}
\end{lem}
%%%%%%%%%%%%%%%%%%%%
%%%%%%%%%%%%%%%%%%%%
\begin{proof}   
We define the polynomials $R_{m}$, $\widetilde{R}_{m}$, $1\leq m\leq M$, and $Q_M$ inductively. For $M$=1 we set
\[
R_1(X_1) := -A X_1, \quad \widetilde{R}_1(X_1) \equiv 0 \quad \text{and} \quad Q_1(X_1 ,T) := bX_1.
\]
For $M\geq 2$ suppose tha thet linear forms $R_{m}(X_1,\dots,X_{m})$, $\widetilde{R}_{m}(X_1,\ldots,X_{m-1})$, $1\leq m\leq M-1$, and the polynomial
\[
Q_{M-1}(X_1,\ldots,X_{M-1},T) = E_{M-1}(X_1,\ldots,X_{M-1})T^{M-2} +\dots+ E_1(X_1,\ldots,X_{M-1}),
\]
say, satisfying \eqref{R} are already defined. We put
\[
R_M(X_1,\ldots,X_M) := (-1)^M\left( A^MX_M + A E_{M-1}(X_1,\ldots,X_{M-1})\right),
\]
\[
\widetilde{R}_M(X_1,\ldots,X_{M-1}) := (-1)^MA E_{M-1}(X_1,\ldots,X_{M-1})
\]
and
\[
\begin{split}
Q_M(X_1,\ldots,X_M,T)&:=Q_{M-1}(X_1,\ldots,X_{M-1},T)T(AT+b+M-1) \\ &+X_M(AT+b)_M - (-1)^M R_M(X_1,\ldots,X_M)T^M.
\end{split}
\]
Then \eqref{Rm} holds,
\[
\deg_{X_j}R_{m}(X_1,\ldots,X_{m}), \deg_{X_j}Q_M(X_1,\ldots,X_M,T)\leq 1\quad \text{for $j=1,\ldots,M$}, 
\]
\[\deg_TQ_M(X_1,\ldots,X_M,T)\leq M-1\]
 and (\ref{R}) holds. The lemma now follows. \hfill
 \end{proof}
 %%%%%%%%%%%%%%%%%%%%
 %%%%%%%%%%%%%%%%%%%%
\begin{lem}
\label{prodgamma}
Let $N\geq 2$, $a_1,\ldots,a_N\in\CC$ and $\lambda_1,\ldots,\lambda_N>0$ be such that $\sum_{\nu=1}^N\lambda_\nu=1$. Moreover, let $a=\sum_{\nu=1}^N a_\nu.$ Then for  $|\arg(w)|\leq \pi-\delta$ with arbitrary $0<\delta<\pi$ and $\min_{1\leq\nu\leq N}|\Im(a_\nu-\lambda_\nu w)|\geq 1$, and any integer $M\geq1$, we have
\begin{equation}\label{prodGamma}
\begin{split}
\prod_{\nu=1}^N\Ga(a_\nu-\lambda_\nu w) &= (2\pi)^{\frac{N-1}{2}} \prod_{\nu=1}^N \lambda_\nu^{a_\nu-\lambda_\nu w -\frac{1}{2}} \sum_{m=0}^M P_m(a_1,\ldots,a_N) \Ga\big(a-\frac{N-1}{2}-w-m\big)\\
& +
O\Big(\frac{1}{|w|}|\Ga(a-\frac{N-1}{2} -w-M)|\Big),
\end{split}
\end{equation}
where the coefficients of $P_m\in\CC[x_1,\ldots,x_N]$ depend on the $\lambda_\nu$'s. Moreover, $P_0\equiv 1$ and for $m\geq 1$
\begin{equation}
\label{Pm}
P_m(a_1,\ldots,a_N)= \frac{1}{2^mm!}\Big(\sum_{\nu=1}^N \frac{a_\nu^2}{\lambda_\nu}-a^2\Big)^m+ \widetilde{P}_m(a_1,\ldots,a_N)
\end{equation}
with $\deg \widetilde{P}_m\leq 2m-1$. The implied constant in \eqref{prodGamma} depends on $\delta,M,$ the $\lambda_\nu$'s and the $a_\nu$'s.
\end{lem}
 %%%%%%%%%%%%%%%%%%%%
  %%%%%%%%%%%%%%%%%%%%
\begin{proof}
Using \eqref{St1} we have 
\[
\begin{split}
\log \prod_{\nu=1}^N\Ga(a_\nu-\lambda_\nu w)&= 
\sum_{\nu=1}^N\left\{
(a_\nu-\frac{1}{2} -\lambda_\nu w)\log(-w) +\lambda_\nu w + \frac{1}{2}\log(2\pi)\right.\\ &\left. \quad +
(a_\nu-\frac{1}{2} -\lambda_\nu w)\log\lambda_\nu - \sum_{j=1}^M\frac{B_{j+1}(a_\nu)}{j(j+1)} \frac{1}{(\lambda_\nu w)^j}
\right\} + O(|w|^{-M-1})\\
&= (a-\frac{N}{2}-w) \log(-w) +w +\frac{N}{2}\log(2\pi) \\ &\quad + 
\sum_{\nu=1}^N (a_\nu-\frac{1}{2})\log\lambda_\nu - \big(\sum_{\nu=1}^N\lambda_\nu \log \lambda_\nu\big) w \\
&- \sum_{\nu=1}^N\sum_{j=1}^M\frac{1}{j(j+1)} \frac{B_{j+1}(a_\nu)}{(\lambda_\nu w)^j}  + O(|w|^{-M-1})\\
& =  \log\Ga(a-\frac{N-1}{2}-w) +\frac{N-1}{2}\log(2\pi) + \sum_{\nu=1}^N (a_\nu-\frac{1}{2})\log\lambda_\nu\\ & \quad - \big(\sum_{\nu=1}^N\lambda_\nu \log \lambda_\nu\big) w 
+ \sum_{j=1}^M \frac{Q_j(a_1,\ldots,a_N)}{w^j}  + O(|w|^{-M-1}),
\end{split}
\]
where the polynomials $Q_j\in\CC[a_1,\ldots,a_N]$ are given by
\begin{equation}
\label{Qj}
Q_j(a_1,\ldots,a_N)= \frac{1}{j(j+1)}B_{j+1}\Big(a-\frac{N-1}{2}\Big)
- \frac{1}{j(j+1)}\sum_{\nu=1}^N \frac{B_{j+1}(a_\nu)}{\lambda_\nu ^j}.
\end{equation}
Taking the exponential of both sides and then expanding into to power series we obtain
\begin{equation}
\label{pg}
\begin{split}
\prod_{\nu=1}^N\Ga(a_\nu-\lambda_\nu w) &= (2\pi)^{\frac{N-1}{2}} \prod_{\nu=1}^N \lambda_\nu^{a_\nu-\lambda_\nu w -\frac{1}{2}} \Ga\big(a-\frac{N-1}{2}-w\big)\\
& \hspace{3cm}\times \exp\Big( \sum_{j=1}^M \frac{Q_j(a_1,\ldots,a_N)}{w^j}  + O(|w|^{-M-1})\Big)\\
 &= (2\pi)^{\frac{N-1}{2}} \prod_{\nu=1}^N \lambda_\nu^{a_\nu-\lambda_\nu w -\frac{1}{2}}  \Ga\big(a-\frac{N-1}{2}-w\big)\\
&\hspace{3cm}\times \left( \sum_{k=0}^M\frac{V_k(a_1,\ldots,a_N)}{w^k} + O(|w|^{-M-1})\right),
\end{split}
\end{equation}
where $V_0\equiv 1$ and for $k\geq 1$
\begin{equation}
\label{Vk}
V_k(a_1,\ldots,a_N)=\sum_{m=1}^k\frac{1}{m!}\sum_{j_1\geq 1}\ldots
\sum_{\substack{j_m\geq 1\\ \hspace{-1.3cm} j_1+\ldots+j_m=k}}\prod_{p=1}^m Q_{j_p}(a_1,\ldots,a_N).
\end{equation}
Applying Lemma \ref{SumXT} with 
\[ 
X_k=V_k(a_1,\ldots,a_N), \quad T=w,\quad A=1\quad \text{and} \quad b=\frac{3-N}{2}-a 
\]
we obtain
\[
\begin{split}
&\Ga\big(a-\frac{N-1}{2}-w\big)\left( 
\sum_{k=0}^M\frac{V_k(a_1,\ldots,a_N)}{w^k} +O\Big(\frac{1}{|w|^{M+1}}\Big)\right) \\
&=\Ga\big(a-\frac{N-1}{2}-w\big)\left(1+ \sum_{m=1}^M\frac{(-1)^m R_m(V_1(a_1,\ldots,a_N),\ldots,V_m(a_1,\ldots,a_N))}{(w+\frac{3-N}{2}-a)_m} + O\Big(\frac{1}{|w|^{M+1}}\Big)\right)\\
&=\sum_{m=1}^M R_m(V_1(a_1,\ldots,a_N),\ldots,V_m(a_1,\ldots,a_N)) \Ga\big(a-\frac{N-1}{2}-w-m\big)\\
 &\hspace{9.6cm} +
O\Big(\frac{1}{|w|}|\Ga(a-\frac{N-1}{2} -w-M)|\Big),
\end{split}
\]
where $R_0\equiv1$. Inserting this into \eqref{pg} we obtain \eqref{prodGamma} with
\begin{equation}
\label{Pm1}
 P_m(a_1,\ldots,a_N)= R_m(V_1(a_1,\ldots,a_N),\ldots,V_m(a_1,\ldots,a_N)).
\end{equation}
$0\leq m\leq M$. It remains to prove \eqref{Pm}. For simplicity we denote by $\Omega_m$ the vector space of polynomials $P\in\CC[a_1,\ldots,a_N]$ with $\deg P\leq m$. Since the $n$-th Bernoulli polynomial has degree $n$ and leading coefficient $1$, from \eqref{Qj} we deduce that for $j\geq 1$ 
\begin{equation}
\label{Qj1}
Q_{j}(a_1,\ldots,a_N)\equiv \frac{1}{j(j+1)}\left(a^{j+1}-\sum_{\nu=1}^N\frac{a_\nu^{j+1}}{\lambda_\nu^j}\right) (\bmod\, \Omega _j)
\end{equation}
and
\[
Q_j(a_1,\dots,a_N) \equiv 0 \ (\bmod\, \Omega_{j+1}).
\]
Consequently, for $k, m\geq 1$ we have
\[
\sum_{j_1\geq 1}\ldots
\sum_{\substack{j_m\geq 1\\ \hspace{-1.3cm} j_1+\ldots+j_m=k}}\prod_{p=1}^m Q_{j_p}(a_1,\ldots,a_N)
\equiv 0 \ (\bmod\,\Omega_{k+m}),
\]
and hence
\[
\sum_{m=1}^{k-1}\frac{1}{m!}\sum_{j_1\geq 1}\ldots
\sum_{\substack{j_m\geq 1\\ \hspace{-1.3cm} j_1+\ldots+j_m=k}}\prod_{p=1}^m Q_{j_p}(a_1,\ldots,a_N)
\equiv 0 \ (\bmod\,\Omega_{2k-1}).
\]
If $m=k$ in \eqref{Vk}, then $j_1=\ldots=j_m=1$, and the corresponding term is
\[
\frac{1}{k!} Q_1^k(a_1,\ldots,a_N).
\]
Thus, recalling \eqref{Qj1}, we have
\[
V_k(a_1,\ldots,a_N)\equiv \frac{1}{2^kk!}\left(a^2-\sum_{\nu=1}^N\frac{a_\nu^2}{\lambda_\nu}\right)^k\  (\bmod\, \Omega_{2k-1})
\]
and
\[
V_k(a_1,\ldots,a_N) \equiv 0 \ (\bmod\, \Omega_{2k}).
\]
Hence from \eqref{Pm1} and \eqref{Rm} we obtain
\[
\begin{split}
P_m(a_1,\ldots,a_N)&\equiv (-1)^mV_m(a_1,\ldots,a_N) \ (\bmod\,\Omega_{2m-1})\\
&\equiv\frac{1}{2^mm!}\Big(\sum_{\nu=1}^N\frac{a_\nu^2}{\lambda_\nu}- a^2\Big)^m \ (\bmod\,\Omega_{2m-1}),
\end{split}
\]
and \eqref{Pm} follows. The proof is complete.
\hfill
\end{proof}
%%%%%%%%%%%%%%%%%%%%
%%%%%%%%%%%%%%%%%%%%
\begin{lem}
\label{Mellin}
For $0<c<\Re (\xi)$ and $|\arg \eta |<\pi$ we have
\[
\frac{1}{2\pi i} \int_{(c)} \Gamma(\xi-w) \Gamma(w) \eta^{-w} \d w = \Gamma(\xi) (1+\eta)^{-\xi}.
\]
\end{lem}
%%%%%%%%%%%%%%%%%%%%
%%%%%%%%%%%%%%%%%%%%
\begin{proof}
See (3.3.9) in Chapter 3 of Paris-Kaminski \cite{Pa-Ka/2001}.
\hfill
\end{proof}
%%%%%%%%%%%%%%%%%%%%
%%%%%%%%%%%%%%%%%%%%
\begin{lem}
\label{Vitali}
 {\rm (Vitali convergence theorem)} For $X\geq X_0$ let $f_X(z_1,\ldots,z_N)$ be holomorphic functions on a domain $\D\subset \CC^N$ such that $|f_X(z_1,\ldots,z_N)| \leq M$ for every $(z_1,\ldots,z_N)\in\D$. Suppose that $f_X(z_1,\ldots,z_N)$ tends to a limit, as $X\to\infty$, on a set of points having an accumulation point in $\D$. Then $f_X(z_1,\ldots,z_N)$ tends to a limit uniformly on any domain $\D'$ whose closure is contained in $\D$, and hence such a limit is holomorphic and bounded by $M$ on $\D'$.
 \end{lem}
 %%%%%%%%%%%%%%%%%%%%
 %%%%%%%%%%%%%%%%%%%%
 \begin{proof}
See Proposition 7 on p.9 of Narasimhan \cite{Nar/1971}.
\hfill
\end{proof}
%%%%%%%%%%%%%%%%%%%%
%%%%%%%%%%%%%%%%%%%%
\begin{lem}
\label{linindep}
The functions $f_j:\CC^N\to\CC$, $j=1,\dots,J$, of the form $f_j(\bfs)=\exp(L_j(\bfs))$ with distinct linear forms $L_j\in \CC[s_1,\dots,s_N]$, are linearly independent over $\CC$.
\end{lem}
%%%%%%%%%%%%%%%%%%%%
%%%%%%%%%%%%%%%%%%%%
\begin{proof}
We proceed by induction on $N$. If $N=1$, then $L_j(s)=a_j s$ for $j=1,\dots,J$, with $a_i\neq a_j$ for $i\neq j$. Suppose that
\[ 
\sum_{j=1}^J A_j \exp(a_j s)\equiv 0
\]
for certain $A_1,\dots,A_N\in\CC$. Differentiating $J-1$ times and then putting $s=0$ we obtain
\[
\sum_{j=1}^J a_j^k A_j =0\quad \text{for} \quad k=0,\dots,J-1.
\]
Since $\det[a_j^k]\neq 0$, we deduce that $A_1=\ldots =A_J=0$, and the case $N=1$ follows. For $N\geq 2$ let 
\[ 
L_j(\bfs) = \sum_{\nu=1}^N a_{j\nu}s_\nu \quad \text{for} \quad j=1,\dots,J,
\] 
and again suppose that
\[
\sum_{j=1}^J A_j \exp(L_j(\bfs))\equiv 0
\]
for certain $A_1,\dots,A_N\in\CC$. Let $b_1,\dots,b_m$, where $m\leq J$, be the distinct numbers in the sequence $a_{1 N},\dots,a_{J N}$. We rewrite the last identity as follows
\[
\sum_{k=1}^m\Big(\sum_{\substack{1\leq j\leq J\\ a_{j N}=b_k}} A_j \exp\Big( \sum_{\nu=1}^{N-1} a_{j\nu}s_\nu\Big)\Big) \exp(b_k s_N)\equiv 0.
\]
Since the assertion is already proved for $N=1$ we deduce that
\[
\sum_{\substack{1\leq j\leq J\\ a_{j N}=b_k}} A_j \exp\Big( \sum_{\nu=1}^{N-1} a_{j\nu}s_\nu\Big)\equiv 0 \quad \text{for} \quad k=1,\dots,m.
\]
Moreover, the above linear forms in $N-1$ variables are obviously distinct since the corresponding forms $L_j$ have the same $N$-th coefficient. Now the lemma follows applying the inductive assumption. 
\hfill
\end{proof}
%%%%%%%%%%%%%%%%%%%%

\medskip
%%%%%
\section{Proof of Theorem \ref{T1}} 
%%%%%%%%%%%%%%%%%%%%%%%%%%%%%%%%%%%%%%%%%%%%%%%%%%%%%%%%%%%%
For $X\geq 1$ and $\alpha>0$ we write
\begin{equation}
\label{new101}
z_X(\alpha)= 2\pi \alpha \omega_X\quad \text{with\quad $\omega_X=\frac{1}{X} +i$}
\end{equation}
and
\begin{equation}
\label{GX}
\bfF_X(\bfs,\alpha)= \sum_{n_1=1}^\infty\ldots\sum_{n_N=1}^\infty
\frac{a_1(n_1)\ldots a_N(n_N)}{n_1^{s_1}\ldots n_N^{s_N}} \exp(- z_X(\alpha)n_1^{\kappa_1}\ldots n_N^{\kappa_N}).
\end{equation}
By the inequality
\[
\prod_{\nu=1}^N n_\nu^{\kappa_\nu}\geq \frac{1}{N} \sum_{\nu=1}^N n_\nu^{\kappa_\nu}
\]
 we have that for every $\bfs\in \CC^N$
\[ 
\sum_{n_1=1}^\infty\ldots\sum_{n_N=1}^\infty
\prod_{\nu=1}^N\frac{|a_\nu(n_\nu)|}{n_\nu^{\si_\nu}} |\exp(-\alpha z_X(\alpha) n_1^{\kappa_1}\ldots n_N^{\kappa_N})| \leq \prod_{\nu=1}^N \sum_{n_\nu=1}^\infty \frac{|a_\nu(n_\nu)|}{n_\nu^{\si_\nu}} \exp\left(-\frac{2\pi \alpha n_\nu^{\kappa_\nu}}{XN}\right)  \ll 1.
\]
Hence, the series in \eqref{GX} is absolutely convergent for every $\bfs\in\CC^N$, and hence $\bfF_X(\bfs,\alpha)$ is entire. Moreover, for $\bfs\in \CC_1^N$ we have
\begin{equation}
\label{limit} 
\lim_{X\to\infty} \bfF_X(\bfs,\alpha) =   \bfF( \bfs, \alpha).
\end{equation}
Using a well-known Mellin-Barnes integral, for
\[
c> \max\left\{ 0, \frac{1-\si_1}{\kappa_1},\ldots,\frac{1-\si_N}{\kappa_N}\right\}
\]
we have
\begin{equation}
\label{GXint}
 \bfF_X(\bfs,\alpha) = \frac{1}{2\pi i} \int_{c-i\infty}^{c+i\infty} \prod_{\nu=1}^N F_\nu(s_\nu+\kappa_\nu w) \Gamma(w) z_X(\alpha)^{-w}\, \d w.
 \end{equation}
Let $\D$ be a bounded domain inside $\{s\in\CC: \si\leq 3/2\}^N\subset \CC^N$ having a non-empty  intersection with  $\CC_1^N$ and write
\[
\|\D\| = \sup\{ |s_\nu|: \bfs=(s_1,\ldots,s_N)\in\D\}.
\]
We also assume that $\D$ is so large that
\begin{equation}
\label{Dbig}
\|\D\| \geq \max_{1\leq\nu\leq N}|\tau_{F_\nu}|+ \max_{1\leq\nu\leq N} |\theta_{F_\nu}| +1.
\end{equation}
Moreover, let $K$ and $V$ be two parameters satisfying
\begin{equation}
\label{KV}
V\geq K\geq 2(\|\D\|+4)\sum_{\nu=1}^N\frac{1}{\kappa_\nu} \quad \text{and} \quad K=[K]+\frac{1}{2}.
\end{equation}

\smallskip
 %%%%%
We change the path of integration in \eqref{GXint} to the broken line consisting of the following three segments
\[ 
\LL_1=(c_1(\bfs)-i\infty, c_1(\bfs)+iK],\ \LL_2=(c_1(\bfs)+i K, -K+iK], \ \LL_3=(-K+i K, -K+i\infty),
\]
where 
\[
c_1(\bfs)= \sum_{\nu=1}^N\frac{2-\si_\nu}{\kappa_\nu}.
\]
Since there are no singularities between $(c-i\infty,c+i\infty)$ and $\LL_1\cup\LL_2\cup\LL_3$, we have
\begin{equation}
\label{GX2B}
\begin{split}
\bfF_X(\bfs,\alpha) &=\frac{1}{2\pi i}\left(\int_{\LL_1} +\int_{\LL_2}+\int_{\LL_3}\right)
\prod_{\nu=1}^N F_\nu(s_\nu+\kappa_\nu w) \Gamma(w) z_X(\alpha)^{-w}\, \d w\\
&=\bfF_X^{(1)}(\bfs,\alpha)+\bfF_X^{(2)}(\bfs,\alpha)+\bfF_X^{(3)}(\bfs,\alpha),
\end{split}
\end{equation}
say. We deal with these terms separately.

\smallskip
%%%%%
For $w=c_1(\bfs)+iv\in \LL_1$ we have $\Re(s_\nu+\kappa_\nu w)\geq 2$ for $\nu=1,\dots,N$, hence
\[
\prod_{\nu=1}^N F_\nu(s_\nu+\kappa_\nu w)\ll 1.
\]
Moreover, by Stirling's formula, $\Gamma(w)\ll (|v|+1)^{c_1(\bfs)-\frac{1}{2}} e^{-\frac{\pi}{2}|v|}$ and also $|z_X(\alpha)^{-w}|\ll e^{\frac{\pi}{2}v}$. Thus
\begin{equation}
\label{GX1}
\bfF_X^{(1)}(\bfs,\alpha)\ll 1
\end{equation}
uniformly for $\bfs\in\D$ and $X\geq 1$, with an implicit constant depending on $\D$ and $\bfF$. Recalling \eqref{KV}, for $\bfs\in\D$ and $w\in\LL_2$ we have $\Im(s_\nu+\kappa_\nu w)\geq \kappa_\nu K-\|\D\|\geq 1$. Moreover, $w$ and all variables $s_\nu+\kappa_\nu w$ are bounded, hence the integrand in $\bfF_X^{(2)}(\bfs,\alpha)$ is bounded as well. Thus
\begin{equation}
\label{GX2}
\bfF_X^{(2)}(\bfs,\alpha)\ll 1
\end{equation}
with the same uniformity as in \eqref{GX1}.

\smallskip
%%%%%
Estimating $\bfF_X^{(3)}(\bfs,\alpha)$ requires more care. Thanks to \eqref{KV}, for $\bfs\in\D$ and $w\in\LL_3$ we have $\Re(s_\nu+\kappa_\nu w)<0$ for $\nu=1,\dots,N$. Applying the invariant functional equation \eqref{EFA} to each $F_\nu$ in 
$\bfF_X^{(3)}(\bfs,\alpha)$, expanding $\overline{F_\nu}(1-s_\nu-\kappa_\nu w)$ into a Dirichlet series and changing the order of summation and integration, we obtain
\begin{equation}
\label{GX3I}
\bfF_X^{(3)}(\bfs,\alpha)=\omega_\bfF \sum_{n_1=1}^\infty \dots \sum_{n_N=1}^\infty \prod_{\nu=1}^N 
\frac{\overline{a_\nu(n_\nu)}}{n_\nu^{1-s_\nu}} \J_X(\bfs,\bfn, \alpha)
\end{equation}
with $\omega_\bfF$ as in \eqref{omegatimes} and
\begin{equation}
\label{J}
\J_X(\bfs,\bfn,\alpha)=
\frac{1}{2\pi i} \int_{\LL_3} \left(\prod_{\nu=1}^N\S_{F_\nu}(s_\nu+\kappa_\nu w) h_{F_\nu}(s_\nu+\kappa_\nu w)\right)\Gamma(w) \left(z_X(\alpha)\prod_{\nu=1}^N n_{\nu}^{-\kappa_\nu}\right)^{-w}\, \d w.
\end{equation}

\smallskip
%%%%%
Suppose first that $\bfn$ is such that
\begin{equation}
\label{CaseA}
\alpha_\bfn >\alpha,
\end{equation} 
where $\alpha_\bfn$ is defined in \eqref{Cn}. In this case, we change the path of integration in \eqref{J} from $\LL_3$ to $\widetilde{\LL}_3=(-\infty+iK, -K+iK)$. The change can be justified using Stirling's formula and Lemma \ref{hasym} with $M=0$. By the same lemma and \eqref{sumdkappa}, for $w\in \widetilde{\LL}_3$ we have, with obvious notation, that
\[ 
\prod_{\nu=1}^N h_{F_\nu}(s_\nu+\kappa_\nu w) \ll  \prod_{\nu=1}^N \left(\frac{q_\nu}{d_\nu^{d_\nu}}\right)^{\kappa_\nu|u|} (2\pi)^{-|u|} \prod_{\nu=1}^N |\Gamma\big(d_\nu(s_{0, \nu}^*-s_\nu-\kappa_\nu w)\big)|,
 \]
where, according to \eqref{s*ell}, $s_{0, \nu}^*= (d_\nu+1)/(2d_\nu) - i\theta_\nu$. Moreover, writing $u=\Re(w)$, by Stirling's formula we get
\[ 
\prod_{\nu=1}^N |\Gamma\big(d_\nu(s_{0,\nu}^*-s_\nu-\kappa_\nu w)\big)| |\Gamma(w)| \ll \prod_{\nu=1}^N (d_\nu \kappa_\nu)^{d_\nu\kappa_\nu |u|} |u|^A
\]
for a certain constant $A$. Further, $z_X(\alpha)^{-w}\ll  (2\pi\alpha(1+O(1/X))^{|u|}$. Inserting the above estimates into \eqref{J} we obtain
\[
\J_X(\bfs,\bfn,\alpha)\ll\int_K^\infty \left(\frac{\alpha_\bfn}{\alpha(1+O(1/X))}\right)^{-u} u^A\, \d u
\]
uniformly for $\bfs\in\D$. Recalling \eqref{CaseA} and \eqref{KV}, we see that for $X\geq X_0=X_0(\alpha)$, where $X_0$ will not necessarily be the same at each occurrence later on, this integral converges quickly and is bounded by
\[
\begin{split}
&\ll \left(\frac{\alpha_\bfn}{\alpha(1+O(1/X)}\right)^{-K/2}\int_K^\infty \left(\frac{\alpha_\bfn}{\alpha(1+O(1/X)}\right)^{-u/2} u^A\, du\\
&\ll \left(\frac{\alpha_\bfn}{\alpha(1+O(1/X)}\right)^{-K/2}\ll \prod_{\nu=1}^N n_\nu^{-\frac{1}{2}\kappa_\nu K}.
\end{split}
\]
Therefore, the part of the sum in \eqref{GX3I} satisfying \eqref{CaseA} is, uniformly for $\bfs\in\D$ and $X$ sufficiently large,
\[
\ll \sum_{n_1=1}^\infty \ldots \sum_{n_N=1}^\infty \prod_{\nu=1}^N 
|a_\nu(n_\nu)| n_\nu^{|s_\nu|-\frac{1}{2}\kappa_\nu K-1}\ll 1,
\]
where the second inequality holds since trivially $a_\nu(n_\nu)\ll n_\nu^2$ and, from  \eqref{KV}, we have 
\[|s_\nu|-\frac{1}{2}\kappa_\nu K\leq -3.\] 

\smallskip
%%%%%%
Suppose now that
\begin{equation}\label{CaseB}
\alpha_\bfn< \alpha.
\end{equation}
Then we use \eqref{Shs1-s} and change the path of integration in \eqref{J} to $(-K+iK,\infty+iK)$. Applying Lemma \ref{hasym} and Stirling's formula in a similar way as before, we obtain 
\[
\begin{split}
\J_X(\bfs,\bfn,\alpha)&\ll\int_{-K}^\infty 
\prod_{\nu=1}^N |\overline{h_{F_\nu}}(1-s_\nu-\kappa_\nu (u+iK))|^{-1}|\Gamma(u+iK)| |z_X(\alpha)|^{-u}
\prod_{\nu=1}^N n_{\nu}^{\kappa_\nu u}\, \d u\\
&\ll \int_{-K}^\infty  \left(\frac{\alpha_\bfn}{\alpha(1+O(1/X))}\right)^{u} (|u|+1)^A\, \d u\ll 1.
\end{split}
\]
Since the number of terms in \eqref{GX3I} with indices satisfying \eqref{CaseB} is finite, this suffices to conclude that their contribution is bounded with the same uniformity as before. Hence the contribution of the terms in \eqref{GX3I} with $\alpha_\bfn\neq\alpha$ is
\begin{equation}
\label{GXX3}
\ll 1
\end{equation}
uniformly for $\bfs\in\D$ and $X\geq X_0$, with an implicit constant depending on $\D, \bfF$ and $\alpha$.

\smallskip
%%%%%
When $\alpha\not\in {\rm Spec}(\bfF)$, all the non-vanishing terms in $\bfF^{(3)}_X(\bfs,\alpha)$ satisfy \eqref{CaseA} or \eqref{CaseB}. Hence recalling \eqref{GX2B}--\eqref{GX2} and \eqref{GXX3} we have
\[
\bfF_X(\bfs,\alpha)\ll 1
\]
with the same uniformity as before. Therefore, by Vitali convergence theorem, see Lemma \ref{Vitali}, we conclude that the limit in \eqref{limit}  exists for all $\bfs\in\D$ and is holomorphic in this region. This gives the analytic continuation of $\bfF( \bfs, \alpha)$ to $\D$, and since $\D$ can be taken arbitrarily large, the first part of Theorem \ref{T1} follows.

\smallskip
%%%%%
Suppose now that $\alpha\in {\rm Spec}(\bfF)$, so there is a non-empty set of $\bfn$'s such that $\alpha_\bfn=\alpha$. 
Observe that if $\bfn$ and $\bfn'$ are such that $\alpha_\bfn=\alpha_{\bfn'}$ then $\J_X(\bfs,\bfn,\alpha)=\J_X(\bfs,\bfn',\alpha)$. Thus, denoting by
\begin{equation}
\label{new-222}
\J_X(\bfs,\alpha) = \frac{1}{2\pi i} \int_{\LL_3} \left(\prod_{\nu=1}^N\S_{F_\nu}(s_\nu+\kappa_\nu w) h_{F_\nu}(s_\nu+\kappa_\nu w)\right)\Gamma(w) \left(z_X(\alpha) \alpha^{-1}\prod_{\nu=1}^N (q_\nu \kappa_\nu^{d_\nu})^{-\kappa_\nu}\right)^{-w}\, \d w
\end{equation}
the common function corresponding to the $\bfn$'s with $\alpha_\bfn=\alpha$, from \eqref{GX2B},\eqref{GX1},\eqref{GX2},\eqref{GX3I} and \eqref{GXX3} we have
\begin{equation}
\label{GXT2}
\bfF_X(\bfs,\alpha) =  \Xi(\bfs,\alpha) \J_X(\bfs,\alpha) + H_X^{(1)}(\bfs,\alpha),
\end{equation}
where $\Xi(\bfs,\alpha)$ is defined by \eqref{Xi} and $H_X^{(1)}(\bfs,\alpha)$ is holomorphic and uniformly bounded as before. Using $\sin(s)=(e^{is}-e^{-is})/(2i)$ and \eqref{omegaxi} we check that there exists a positive $\lambda_0$ such that
\begin{equation}
\label{new100}
\prod_{\nu=1}^N S_{F_\nu}(s_\nu+\kappa_\nu w) = e^{-iA(\bfs)- i\frac{\pi}{2}w}\left(1+ O(e^{-\lambda_0v})\right)
\end{equation}
for $w \in \LL_3$, where $v=\Im(w)$ and, with the notation in \eqref{omegaxi},
\begin{equation}
\label{Abfs}
A(\bfs)= \frac{\pi}{2} \sum_{\nu=1}^N(d_\nu s_\nu +\xi_{\nu}). 
\end{equation}
Inserting \eqref{new100} into \eqref{new-222} we obtain
\begin{equation}
\label{new3}
\J_X(\bfs, \alpha)=\I_X(\bfs,\alpha) + H_X^{(2)}(\bfs,\alpha),
\end{equation}
where
\begin{equation}
\label{J1}
\I_X(\bfs,\alpha)=e^{-iA(\bfs)}\frac{1}{2\pi i} \int_{\LL_3} \prod_{\nu=1}^N h_{F_\nu}(s_\nu+\kappa_\nu w)\Gamma(w)
\left(iz_X(\alpha) \alpha^{-1}\prod_{\nu=1}^N (q_\nu \kappa_\nu^{d_\nu})^{-\kappa_\nu}\right)^{-w}\, \d w.
\end{equation}
Here $H_X^{(2)}(\bfs,\alpha)$ is holomorphic and uniformly bounded as before, as it can easily be justified using the error term $O(e^{-\lambda_0v})$ in \eqref{new100}.

\smallskip
%%%%%
Recalling \eqref{new1}, let $M\geq0$ be an integer such that
\begin{equation}
\label{M}
 (\|\D\|+\frac{1}{2}) d -\frac{1}{2} < M \leq (\|\D\|+\frac{1}{2})d +\frac{1}{2}.
 \end{equation}
We write $s_{\ell,\nu}^*=\frac{d_\nu+1}{2d_\nu}-\frac{\ell}{d_\nu}-i\theta_\nu$ and apply Lemmas \ref{hasym} and \ref{prod} with
\[ 
\alpha_\nu= \frac{1}{\sqrt{2\pi}} \left(\frac{q_\nu^{1/d_\nu}}{2\pi d_\nu}\right)^{d_\nu(\frac12-s_\nu-\kappa_\nu w)} \sum_{\ell=0}^M d_{F_\nu}(\ell) \Gamma\big(d_\nu(s_{\ell,\nu}^*-s_\nu - \kappa_\nu w)\big),
\]
\[ 
r_\nu = h_{F_\nu}(s_\nu+\kappa_\nu w) - \alpha_\nu \quad \text{and}\quad  A_\nu= c(\D) |\Ga(d_\nu(s_{0,\nu}^*-s_\nu-\kappa_\nu w)|,
 \]
where $c(\D)>0$ is a sufficiently large constant. We obtain that
\begin{equation}
\label{prodh}
\prod_{\nu=1}^N h_{F_\nu}(s_\nu+\kappa_\nu w) = B_1(\bfs) E_1^{-w}\sum_{\ell_1=0}^M\ldots\sum_{\ell_N=0}^M
\prod_{\nu=1}^N d_{F_\nu}(\ell_\nu) \Gamma\big(d_\nu(s_{\ell_\nu,\nu}^*-s_\nu-\kappa_\nu w)\big) + R(\bfs,w),
\end{equation}
say, where
\begin{equation}
\label{B1bfs}
B_1(\bfs)=(2\pi)^{-N/2}\prod_{\nu=1}^N \left(\frac{q_\nu^{1/d_\nu}}{2\pi d_\nu}\right)^{d_\nu(\frac12-s_\nu)},
\qquad
E_1= \prod_{\nu=1}^N\left(\frac{q_\nu^{1/d_\nu}}{2\pi d_\nu}\right)^{d_\nu\kappa_\nu},
\end{equation}
\begin{equation}
\label{Restimate}
R(\bfs,w)\ll \max_{1\leq\nu\leq N} \Big(\frac{1}{|s_\nu+\kappa_\nu w|}
|\Ga\big(d_\nu(s_{M,\nu}^*-s_\nu- \kappa_\nu w)\big)|
\prod_{\substack{1\leq\mu\leq N\\ \mu\neq \nu}} |\Ga\big(d_\mu(s_{0,\nu}^*-s_\mu- \kappa_\mu w)\big)|\Big).
\end{equation}
Recalling \eqref{Dbig} and \eqref{KV}, for $w\in \LL_3$ we have 
\[
|\Im (d_\nu(s_{\ell_\nu,\nu}^*-s_\nu - \kappa_\nu w))|\geq d_\nu( \kappa_\nu V -  \|\D\| - |\theta_\nu| ) \geq 1
\]
and
\[ 
\Im(s_\nu+\kappa_\nu w)\geq \kappa_\nu V- \|\D\| \geq \|\D\| \geq |\theta_\nu|+1.
\]
Hence, for $w\in \LL_3$, the remainder $R(\bfs,w)$ in \eqref{prodh} is holomorphic as a function of $\bfs\in \D$.
Moreover, using \eqref{Restimate} and \eqref{St2}, for such $w$ we have
\[
 R(\bfs,w) \Gamma(w) \ll e^{-\pi v} |w|^A,
\]
where, using  \eqref{M},
\[ 
A\leq (\|\D\|+\frac12)d -M-\frac{3}{2}<-1.
\]
Inserting \eqref{prodh} into \eqref{J1} we therefore obtain that
\begin{equation}
\label{J1B}
\begin{split}
&\I_X(\bfs,\alpha)= e^{-iA(\bfs)}B_1(\bfs)\sum_{\ell_1=0}^M\ldots\sum_{\ell_N=0}^M
\prod_{\nu=1}^N d_{F_\nu}(\ell_\nu)  \frac{1}{2\pi i} \int_{\LL_3} \prod_{\nu=1}^N\Gamma(d_\nu(s_{\ell_\nu,\nu}^*-s_\nu)-d_\nu\kappa_\nu w) \\
&\hskip1cm \times \Gamma(w) \left(iz_X(\alpha)\alpha^{-1}\prod_{\nu=1}^N (2\pi d_\nu\kappa_\nu)^{-d_\nu\kappa_\nu}\right)^{-w}\, \d w+ H_X^{(3)}(\bfs,\alpha),
\end{split}
\end{equation}
where $H_X^{(3)}(\bfs,\alpha)$ is holomorphic and uniformly bounded as before.

\smallskip
%%%%%
Recalling \eqref{new1} and writing $\bfell=\sum_{\nu=1}^N \ell_\nu$ we have
\[
\sum_{\nu=1}^Nd_\nu(s_{\ell_\nu,\nu}^*-s_\nu) -\frac{N-1}{2} = \frac{d+1}{2}-\bfell-i\theta -\sum_{\nu=1}^N d_\nu s_\nu.
\]
Hence, applying Lemma \ref{prodgamma} we deduce that
\begin{equation}
\label{new2}
\begin{split}
&\prod_{\nu=1}^N\Gamma\big(d_\nu(s_{\ell_\nu,\nu}^*-s_\nu)-d_\nu\kappa_\nu w\big) \\ 
&\hspace{1cm}=B_2(\bfs) E_2^{-w}
\sum_{m=0}^M P_{\ell_1,\ldots,\ell_N,m}(\bfs) \Ga\Big(\frac{d+1}{2}-\bfell-m-i\theta - \sum_{\nu=1}^N d_\nu s_\nu -w\Big)\\
&\hspace{1cm}\quad  +O\Big(\frac{1}{|w|}\Big| \Ga\Big(\frac{d+1}{2}-\bfell-M-i\theta - \sum_{\nu=1}^N d_\nu s_\nu -w\Big)\Big|\Big),
\end{split}
\end{equation}
where
\begin{equation}
\label{B2bfs}
B_2(\bfs) = (2\pi)^{\frac{N-1}{2}}\prod_{\nu=1}^N (d_\nu \kappa_\nu)^{d_\nu(\frac{1}{2}- s_\nu-i\theta_\nu)-\ell_\nu}, \qquad 
E_2= \prod_{\nu=1}^N(d_\nu\kappa_\nu)^{d_\nu\kappa_\nu}
\end{equation}
and $P_{\ell_1,\ldots,\ell_N,m}\in \CC[\bfs]$ are polynomials defined by
\begin{equation}
\label{Pl1lN}
P_{\ell_1,\ldots,\ell_N,m}(\bfs) = P_m(a_1,\ldots,a_N),
\end{equation}
where the $P_m$'s are as in \eqref{prodGamma}--\eqref{Pm} with parameters $\lambda_\nu=d_\nu\kappa_\nu$ and 
\begin{equation}
\label{aell}
a_\nu= d_\nu(s_{\ell_\nu,\nu}^*-s_\nu).
\end{equation}
Next we insert \eqref{new2} into \eqref{J1B} and observe that, thanks to \eqref{sumdkappa} and \eqref{new101}, similarly as before we obtain
\begin{equation}
\label{J1C}
\begin{split}
&\I_X(\bfs,\alpha)= e^{-iA(\bfs)}B_1(\bfs)B_2(\bfs)\sum_{\ell_1=0}^M\ldots\sum_{\ell_N=0}^M\sum_{m=0}^M
\prod_{\nu=1}^N \big(d_{F_\nu}(\ell_\nu)(d_\nu\kappa_\nu)^{-\ell_\nu}\big)P_{\ell_1,\ldots,\ell_N,m}(\bfs)\\
&\quad\times  \frac{1}{2\pi i} \int_{\LL_3}\Gamma\Big(\frac{d+1}{2}-\bfell-m-i\theta -\sum_{\nu=1}^Nd_\nu s_\nu -w\Big)\Gamma(w)
(i\omega_X)^{-w}\, \d w
+ H_X^{(4)}(\bfs,\alpha)
\end{split}
\end{equation}
with a certain holomorphic function $H_X^{(4)}(\bfs,\alpha)$, uniformly bounded as before.

\smallskip
%%%%%
Let $\ell\geq-1$ be an integer. If
\begin{equation}\label{SellA}
\sum_{\nu=1}^Nd_\nu\Re(s_\nu) \geq \frac{d+1}{2} -\ell -\frac{2}{3},
\end{equation}
$\bfell+m\geq \ell+1$ and $w\in\LL_3$ we have
\[
\Gamma\Big(\frac{d+1}{2}-\bfell-m-i\theta -\sum_{\nu=1}^Nd_\nu s_\nu -w\Big)\Gamma(w)
\ll e^{-\pi v}|w|^{-\bfell-m+\ell-\frac13}\ll e^{-\pi v} v^{-4/3}.
\]
Therefore, the corresponding integrals in \eqref{J1C} are bounded, and hence we may restrict the range of summation in \eqref{J1C} to $\bfell+m\leq \ell$, thus generating a negligible error term. Recalling \eqref{GXT2},\eqref{new3} and \eqref{J1C}, for any given $\ell\geq-1$ and $\bfs\in \D$ satisfying \eqref{SellA} we can write
\begin{equation}
\label{GX3}
\begin{split}
\bfF_X(\bfs,\alpha) = &\sum_{k=0}^\ell \Xi_k(\bfs,\alpha) \frac{1}{2\pi i} \int_{\LL_3}\Gamma\Big(\frac{d+1}{2}-k-i\theta -\sum_{\nu=1}^Nd_\nu s_\nu -w\Big)\Gamma(w)
(i\omega_X)^{-w}\, \d w\\
&\hspace{8.5cm} + H_{X,\ell}^{(5)}(\bfs,\alpha);
\end{split}
\end{equation}
here $H_{X,\ell}^{(5)}(\bfs,\alpha)$ is holomorphic and bounded uniformly for $\bfs\in \D$ satisfying \eqref{SellA} and $X\geq X_0$, and
\begin{equation}
\label{Xik}
\Xi_k(\bfs,\alpha)=e^{-iA(\bfs)}B_1(\bfs)B_2(\bfs)\Xi(\bfs,\alpha)\sum_{\ell_1=0}^M\dots\sum_{\ell_N=0}^M\sum_{\substack{m=0\\ \\\hspace{-1.6cm} \bfell+m=k}}^M
\prod_{\nu=1}^N \big(d_{F_\nu}(\ell_\nu)(d_\nu\kappa_\nu)^{-\ell_\nu}\big)P_{\ell_1,\dots,\ell_N,m}(\bfs)
\end{equation}
with $\Xi(\bfs,\alpha)$ defined in \eqref{Xi}. Note that for $\ell=-1$ the sum in \eqref{GX3} is empty, and hence such a term vanishes. For $\ell\geq 0$ we set
\[ 
\D_\ell=\Big\{ \bfs\in\D: \frac{d+1}{2} -\ell-\frac23 \leq \sum_{\nu=1}^N d_\nu \Re(s_\nu) <\frac{d+1}{2} -\ell+ \frac23\Big\},
\]
and for $\ell=-1$ let
\[ 
\D_{-1}=\Big\{ \bfs\in\D:  \sum_{\nu=1}^N d_\nu \Re(s_\nu) \geq \frac{d+1}{2} + \frac13\Big\}.
\]
If $\D_\ell \neq \emptyset$ then 
\[ 
\ell< \frac{d+1}{2} -\sum_{\nu=1}^N d_\nu \Re(s_\nu) +\frac23\leq d \big(\frac12 +\|\D\|\big) + \frac76.
\] 
Hence, recalling \eqref{KV}, for $k\leq \ell$, $\Re(w)=-K$ and $\bfs\in \D_\ell$ we have
\[
\Re\Big(\frac{d+1}{2}-k-i\theta -\sum_{\nu=1}^Nd_\nu s_\nu -w\Big) \geq  \frac{d+1}{2}-\ell -\sum_{\nu=1}^Nd_\nu \Re(s_\nu) +K \geq K- 2d\|\D\|-\frac23>7.
\]
Therefore, the integrand in \eqref{GX3} is holomorphic for $\bfs\in\D_\ell$ for every $\Re(w)=-K$. Moreover, for $\Im(w)=v\leq V$ we have 
\[
|(i\omega_X)^{-w}| \ll_V 1
\]
uniformly for $X\geq X_0$. Hence, replacing the path of integration in \eqref{GX3} by the whole vertical line $\Re(w)=-K$ generates an error of size $O(1)$ uniformly for $\bfs\in\D_\ell$ and $X\geq X_0$. Thus, we obtain
\begin{equation}
\label{GX4}
\bfF_X(\bfs,\alpha) = \sum_{k=0}^\ell \Xi_k(\bfs,\alpha) \I_{X,k}(\bfs,\alpha)+ H_{X,\ell}^{(6)}(\bfs,\alpha),
\end{equation}
where
\begin{equation}
\label{Ik}
\I_{X.k}(\bfs,\alpha)=\frac{1}{2\pi i} \int_{(-K)}\Gamma\Big(\frac{d+1}{2}-k-i\theta -\sum_{\nu=1}^Nd_\nu s_\nu -w\Big)\Gamma(w)
(i\omega_X)^{-w}\, \d w 
\end{equation}
and $H_{X,\ell}^{(6)}(\bfs,\alpha)$ is holomorphic and uniformly bounded for $\bfs\in\D_\ell$ and $X\geq X_0$. 

\smallskip
%%%%%
Now we compute $\I_{X,k}(\bfs,\alpha)$ using Lemma \ref{Mellin}. For $k\leq \ell -1$ and $\bfs\in\D_\ell$ we have 
\[
\Re\left(\frac{d+1}{2}-k-i\theta -\sum_{\nu=1}^Nd_\nu s_\nu\right) \geq 
\frac{d+1}{2}-\ell +1 -\sum_{\nu=1}^Nd_\nu\Re(s_\nu)\geq \frac13.
\]
Thus, shifting the line of integration in \eqref{Ik} to $\Re(w)=1/6$ and then using Lemma \ref{Mellin} with
\[
\xi=\frac{d+1}{2}-k-i\theta -\sum_{\nu=1}^Nd_\nu s_\nu, \quad c=\frac16 \quad \text{and}\quad
\eta=i\omega_X,
\]
we obtain
\begin{equation}
\label{IkA}
\begin{split}
\I_{X,k}(\bfs,\alpha) &= 
\Gamma\Big( \frac{d+1}{2}-k-i\theta -\sum_{\nu=1}^Nd_\nu s_\nu\Big)
\big((e^{-i\frac{\pi}{2}} X)^{\frac{d+1}{2}-k -i\theta-\sum_{\nu=1}^Nd_\nu s_\nu} - 1\big)\\
&\hspace{2cm} - \sum_{\mu=1}^{[K]} \frac{(-1)^\mu}{\mu!} \Gamma\left( \frac{d+1}{2}-k-i\theta -\sum_{\nu=1}^N d_\nu s_\nu + \mu\right)(i\omega_X)^\mu\\
& \hskip-1cm =\Gamma\Big( \frac{d+1}{2}-k-i\theta -\sum_{\nu=1}^Nd_\nu s_\nu\Big)
\big((e^{-i\frac{\pi}{2}} X)^{\frac{d+1}{2}-k -i\theta-\sum_{\nu=1}^Nd_\nu s_\nu} - 1\big) + H_{X,k}^{(7)}(\bfs,\alpha),
\end{split}
\end{equation}
where $H_{X,k}^{(7)}(\bfs,\alpha)$ is holomorphic and uniformly bounded for $\bfs\in \D_\ell$ and $X\geq X_0$. Let now $k=\ell$, write
\[
\D_\ell^-=\Big\{ \bfs\in\D_\ell: \sum_{\nu=1}^Nd_\nu\Re(s_\nu) <\frac{d+1}{2} -\ell\Big\}
\]
and observe that for $\bfs\in\D_\ell^-$ we have
\[
\Re\left(\frac{d+1}{2}-\ell-i\theta -\sum_{\nu=1}^Nd_\nu s_\nu\right) \geq  \frac{d+1}{2}-\ell  -\sum_{\nu=1}^Nd_\nu\Re(s_\nu)>0.
\]
Thus, shifting the line of integration in \eqref{Ik} to the right and then using Lemma \ref{Mellin} similarly as before with
\[
\xi=\frac{d+1}{2}-\ell-i\theta -\sum_{\nu=1}^Nd_\nu s_\nu , \quad  c=\frac12\Big(\frac{d+1}{2}-\ell -\sum_{\nu=1}^Nd_\nu\Re(s_\nu)\Big) , \quad \eta=i\omega_X,
\]
we conclude that, for $\bfs\in \D_\ell^-$,
\begin{equation}
\label{Iell}
\begin{split}
\I_{X,\ell}(\bfs,\alpha) &= 
\Gamma\Big( \frac{d+1}{2}-\ell-i\theta -\sum_{\nu=1}^Nd_\nu s_\nu\Big)
\big((e^{-i\frac{\pi}{2}} X)^{\frac{d+1}{2}-\ell -i\theta-\sum_{\nu=1}^Nd_\nu s_\nu} - 1\big)\\
&\hspace{2cm} - \sum_{\mu=1}^{[K]} \frac{(-1)^\mu}{\mu!} \Gamma\left( \frac{d+1}{2}-\ell-i\theta -\sum_{\nu=1}^N d_\nu s_\nu + \mu\right)(i\omega_X)^\mu\\
& \hskip-1cm =\Gamma\Big( \frac{d+1}{2}-\ell-i\theta -\sum_{\nu=1}^Nd_\nu s_\nu\Big)
\big((e^{-i\frac{\pi}{2}} X)^{\frac{d+1}{2}-\ell -i\theta-\sum_{\nu=1}^Nd_\nu s_\nu} - 1\big) + H_{X,\ell}^{(8)}(\bfs,\alpha),
\end{split}
\end{equation}
say, where $H_{X,\ell}^{(8)}(\bfs,\alpha)$ is clearly holomorphic and uniformly bounded for $\bfs\in \D_\ell$ and $X\geq X_0$. Inserting \eqref{IkA} and \eqref{Iell} into \eqref{GX4} we obtain, for $\bfs\in\D_\ell^-$, that
\begin{equation}
\label{GX5}
\begin{split}
\bfF_X(\bfs,\alpha) &= 
\sum_{k=0}^\ell \Xi_k(\bfs,\alpha) \Gamma\Big( \frac{d+1}{2}-k-i\theta -\sum_{\nu=1}^Nd_\nu s_\nu\Big)
\big((e^{-i\frac{\pi}{2}} X)^{\frac{d+1}{2}-k -i\theta-\sum_{\nu=1}^Nd_\nu s_\nu} - 1\big)\\ 
& \hspace{10cm} + R_{X,\ell}(\bfs,\alpha),
\end{split}
\end{equation}
where $R_{ X,\ell}(\bfs,\alpha)$ is holomorphic and uniformly bounded for $\bfs\in\D_\ell^-$ and $X\geq X_0$.
By analytic continuation this formula holds for all $\bfs\in\D_\ell$, with $R_{ X,\ell}(\bfs,\alpha)$ holomorphic and uniformly bounded on $\D_\ell$ since the $\Xi_k(\bfs,\alpha)$ in \eqref{Xik} are holomorphic on $\D$, and bounded since $\D$ is bounded.

\smallskip
%%%%%
We now prove by a finite induction on $-1\leq \ell\leq L:=\max\{\ell: \D_\ell\neq\emptyset\}$ that the following statements hold true.

(A) For $\bfs\in \D_\ell$ the limit
\begin{equation}
\label{limitR}
 R_{\ell}(\bfs,\alpha) = \lim_{X\to\infty} R_{X,\ell}(\bfs,\alpha)
\end{equation}
exists and is holomorphic on $\D_\ell$.

(B) The standard twist $\bfF(\bfs,\alpha)$ admits analytic continuation to $\D_\ell\backslash H^*_\ell$, where the hyperplanes $H^*_\ell$ with $\ell\geq 0$ are defined in \eqref{H}, whereas $H^*_{-1}=\emptyset$ since for $\ell=-1$ the sum in \eqref{GX3} is empty.

(C) For $\bfs\in\D_\ell$, we have 
\[
R_{\ell}(\bfs,\alpha)= \bfF(\bfs,\alpha) + \sum_{k=0}^\ell\Xi_k(\bfs,\alpha)\Gamma\Big( \frac{d+1}{2}-k-i\theta -\sum_{\nu=1}^Nd_\nu s_\nu\Big).
\]

\noindent
For $\ell=-1$, \eqref{GX5} tells us that for $\bfs\in \D_{-1}$ we have
\[ 
\bfF_X(\bfs,\alpha)=R_{X,-1}(\bfs,\alpha)=O(1)
\]
uniformly for $X\geq X_0$. Recalling \eqref{limit}, we see that the limit in \eqref{limitR} exists for $\bfs\in \D_{-1}\cap \CC_1^N$, and hence by the Vitali convergence theorem, see Lemma 7, such a limit exists for all $\bfs\in\D_{-1}$ and is holomorphic in this domain. This shows (A), and statements (B) and (C) follow as well since $H^*_\ell=\emptyset$ and $\Xi_\ell(\bfs,\alpha)\equiv 0$ for $\ell=-1$. Let now $0\leq\ell\leq L$ and suppose that (A)-(C) hold up to $\ell-1$. According to \eqref{GX5}, for $\bfs\in \D_\ell \cap\D_{\ell-1}$ we have 
\[
\begin{split}
R_{X,\ell}(\bfs,\alpha)&=R_{X,\ell-1}(\bfs,\alpha) \\
&- \Xi_\ell(\bfs,\alpha)\Gamma\Big( \frac{d+1}{2}-\ell-i\theta -\sum_{\nu=1}^Nd_\nu s_\nu\Big)
\big((e^{-i\frac{\pi}{2}} X)^{\frac{d+1}{2}-\ell -i\theta-\sum_{\nu=1}^Nd_\nu s_\nu} - 1\big).
\end{split}
\]
For such $\bfs$, the exponent in the last factor on the right-hand side has negative real part. Thus, using in addition the inductive assumption, we infer that for $\bfs\in \D_\ell \cap\D_{\ell-1}$ the limit $R_\ell(\bfs,\alpha)$ in \eqref{limitR} exists and equals
\begin{equation}\label{R+Xi}
\begin{split}
&R_{\ell-1}(\bfs,\alpha) + \Xi_\ell(\bfs,\alpha)\Gamma\Big( \frac{d+1}{2}-\ell-i\theta -\sum_{\nu=1}^Nd_\nu s_\nu\Big)\\
&\hspace{1cm}=\bfF(\bfs,\alpha) + \sum_{k=0}^\ell\Xi_k(\bfs,\alpha)\Gamma\Big( \frac{d+1}{2}-k-i\theta -\sum_{\nu=1}^Nd_\nu s_\nu\Big).
\end{split}
\end{equation}
By Vitali's theorem, the limit exists for all $\bfs\in\D_\ell$, thus $R_{\ell}(\bfs,\alpha)$ is holomorphic in $\D_\ell$. Therefore, \eqref{R+Xi} gives analytic continuation of $\bfF(\bfs,\alpha)$ to $\D_\ell\backslash H^*_\ell$. Hence conditions (B) and (C) follow, and the inductive argument is complete.

\smallskip
%%%%%
We can finally conclude the proof of Theorem \ref{T1}. Since $\D=\bigcup_{\ell=-1}^{L}\D_\ell$, from (B) we deduce that 
$\bfF(\bfs,\alpha)$  admits analytic continuation to $\D\backslash \bigcup_{\ell=0}^\infty H^*_\ell$ and, moreover, \eqref{limitXiell} holds. Theorem \ref{T1} now follows since $\D$ can be taken arbitrarily large.

\medskip
%%%%%
\section{Proof of Theorem \ref{T2}} 
%%%%%%%%%%%%%%%%%%%%%%%%%%%%%%%%%%%%%%%%%%%%%%%%%%%%%%%%%%%%
A computation based on \eqref{Abfs},\eqref{B1bfs} and  \eqref{B2bfs} shows that
\[
e^{-iA(\bfs)}B_1(\bfs)B_2(\bfs)=  \frac{1}{\sqrt{2\pi}}e^{- i \sum_{\nu=1}^{N} \big(\frac{\pi}{2} d_\nu\, s_\nu + \frac{\pi}{2}\,\xi_\nu + d_\nu\,\theta_\nu\,\log(d_\nu\,\kappa_\nu)\big)} \prod_{\nu=1}^N \Big(\frac{\kappa_\nu q_\nu^{1/d_\nu}}{2\pi}\Big)^{d_\nu\big(\frac{1}{2}-s_\nu\big)}.
\]
Hence \eqref{Xik} implies \eqref{Xiell} with
\begin{equation}
\label{new111}
W_\ell(\bfs)= \sum_{\ell_1\geq0}\dots\sum_{\ell_N\geq0}\sum_{\substack{m\geq0\\ \\\hspace{-1.6cm} \bfell+m=\ell}}
\prod_{\nu=1}^N \big(d_{F_\nu}(\ell_\nu)(d_\nu\kappa_\nu)^{-\ell_\nu}\big)P_{\ell_1,\ldots,\ell_N,m}(\bfs),
\end{equation}
where the polynomials $P_{\ell_1,\ldots,\ell_N,m}(\bfs)$ are defined in \eqref{Pl1lN} and \eqref{aell}. Note that the upper limit $M$ in the summations, present in \eqref{Xik}, is not present in \eqref{new111}. Indeed, $\D$ is a large
domain and $\ell$ is such that $\D\cap H^*_\ell\neq \emptyset$. For $\bfs\in\D\cap H^*_\ell$ we have $\ell\leq d(\|\D\|+\half)+\half$ and hence, recalling \eqref{M}, the condition $\bfell+m=\ell$ implies that $\ell_1,\dots,\ell_N,m\leq M$. Hence, the upper limit $M$ in the summations in \eqref{Xik} can be dropped.

\smallskip
%%%%% 
Next we compute the degree of $W_\ell(\bfs)$. For simplicity we denote by $\widetilde{\Omega}_n$ the subspace of the polynomials in $\CC[\bfs]$ of degree $\leq n$. If $(\ell_1,\ldots,\ell_N)\neq (0,\ldots,0)$  then, according to Lemma \ref{prodgamma}, we have
\[
P_{\ell_1,\ldots,\ell_N,m}(\bfs)\equiv 0 \, (\bmod\,\tilde{\Omega}_{2(\ell-1)}).
\]
Thus, recalling that $|d_{F_\nu}(0)|=1$, see (1.22) in \cite{Ka-Pe/2021a}, and using \eqref{Pl1lN},\eqref{aell} and \eqref{Pm} we deduce that
\begin{equation}
\label{Well2}
W_\ell(\bfs)\equiv P_{0,\ldots,0,\ell}(\bfs) \, (\bmod\,\widetilde{\Omega}_{2(\ell-1)}) \equiv \frac{1}{2^\ell \ell!}\Big(\sum_{\nu=1}^N \frac{d_\nu }{\kappa_\nu}s_\nu^2 -\Big(\sum_{\nu=1}^N d_\nu s_\nu\Big)^2\Big)^\ell \, (\bmod\,\tilde{\Omega}_{2\ell-1}).
\end{equation}
This shows that $\deg W_\ell=2\ell$ since $N\geq 2$.

\smallskip
%%%%%
To conclude the proof we have to check that for every $\ell$ the ``residual function'' $\Xi_\ell(\bfs,\alpha)$ does not vanish identically on the hyperplane $H^*_\ell$. For $\bfs=(s_1,\ldots,s_N)\in H^*_\ell$ we have 
\begin{equation}
\label{sN}
 s_N=-\frac{1}{d_N} \sum_{\nu=1}^{N-1}d_\nu s_\nu+\frac{d+1}{2d_N}-\frac{\ell}{d_N} -i\frac{\theta}{d_N},
 \end{equation}
and hence, according to \eqref{Well2}, we obtain
\[
W_\ell(\bfs)\equiv  \frac{1}{2^\ell \ell!}\Big(\sum_{\nu=1}^{N-1} \frac{d_\nu }{\kappa_\nu}s_\nu^2
+\Big(\sum_{\nu=1}^{N-1}\frac{d_\nu}{d_N} s_\nu\Big)^2\Big)^\ell
\, (\bmod\,\tilde{\Omega}_{2\ell-1}\cap\CC[s_1,\ldots,s_{N-1}]).
\]
Therefore, the restriction $W_\ell(\bfs)_{|H^*_\ell}$ is a polynomial in $\CC[s_1,\ldots,s_{N-1}]$ of degree $2\ell$. In particular, it cannot vanish identically on $H^*_\ell$. Turning to  $\Xi(\bfs,\alpha)$, for $\bfs\in H^*_\ell$ and using \eqref{sN} we rewrite  \eqref{Xi} as
\[
\Xi(\bfs, \alpha) = \omega_\bfF \sum_{\substack{\bfm\in\NN^N\\ \alpha_\bfm=\alpha}} A(\bfm) \exp(L_{\bfm}(\bfs)),
\]
where
\[
A(\bfm)=m_N^{\frac{1}{d_N}\big(\frac{d+1}{2}-\ell-i\theta\big)}\prod_{\nu=1}^N\frac{\overline{a_\nu(m_\nu)}}{m_\nu}
\]
and
\[
L_\bfm(\bfs)= \sum_{\nu=1}^{N-1}\Big(\log m_\nu-\frac{d_\nu}{d_N}\log m_N\Big) s_\nu.
\]
If $\Xi_\ell(\bfs,\alpha)\equiv 0$ on $H^*_\ell$ then, according to Lemma \ref{linindep}, not all linear forms $L_\bfm(\bfs)$ are different. Let $\bfm\neq \bfm'$ be such that $\alpha_\bfm = \alpha_{\bfm'}=\alpha$ and $L_\bfm(\bfs) = L_{\bfm'}(\bfs)$ for all $\bfs\in \CC^{N-1}$. Then we have
\begin{equation}
\label{KmI}
 \sum_{\nu=1}^N \kappa_\nu \log m_\nu =  \sum_{\nu=1}^N \kappa_\nu \log m_\nu',
 \end{equation}
and for $1\leq \nu\leq N-1$
\begin{equation}
\label{KmII}
\log m_\nu-\frac{d_\nu}{d_N}\log m_N = \log m_\nu'-\frac{d_\nu}{d_N}\log m_N'.
\end{equation}
Recalling \eqref{sumdkappa} we have
\[
\begin{split}
\sum_{\nu=1}^N \kappa_\nu \log m_\nu &= \kappa_N\log m_N +\sum_{\nu=1}^{N-1} \kappa_\nu \log m_\nu
=  \kappa_N\log m_N +\sum_{\nu=1}^{N-1}\kappa_\nu \Big(\log m_\nu' +\frac{d_\nu}{d_N}\log \frac{m_N}{m_N'}\Big)\\
&=\sum_{\nu=1}^N \kappa_\nu \log m_\nu' + \frac{1}{d_N}\log\frac{m_N}{m_N'}.
\end{split}
\]
This, together with \eqref{KmI}, gives $m_N=m_N'$. But then \eqref{KmII} implies $m_\nu=m_\nu'$ for $1\leq\nu\leq N-1$. Consequently, $\bfm=\bfm'$, contrary to our assumption. So we proved that all factors on the right-hand side of \eqref{Xiell} do not vanish identically on $H^*_\ell$, and Theorem \ref{T2} follows.

\bigskip
\noindent
Jerzy Kaczorowski, Faculty of Mathematics and Computer Science, A.Mickiewicz University, 61-614 Pozna\'n, Poland. e-mail: \url{kjerzy@amu.edu.pl}

\medskip
\noindent
Alberto Perelli, Dipartimento di Matematica, Universit\`a di Genova, via Dodecaneso 35, 16146 Genova, Italy. e-mail: \url{alberto.perelli@unige.it}

\end{document}